\documentclass{amsart}
\usepackage{graphicx}

\def\lam{\lambda }

\providecommand{\abs}[1]{\lvert#1\rvert}
\providecommand{\norm}[1]{\lVert#1\rVert}
\newcommand{\remove}[1]{ }

\newtheorem{theorem}{Theorem}[section]

\newtheorem{lemma}[theorem]{Lemma}

\newtheorem{conjecture}[theorem]{Conjecture}

\theoremstyle{definition}

\theoremstyle{remark}
\newtheorem*{remark}{Remark}

\numberwithin{equation}{section}

\begin{document}
\title{A Bernstein type inequality} 
%\date{2010-02-28}
\author{Vilmos Komornik}
\address{D\'epartement de Math\'ematique\\
         Universit\'e de Strasbourg\\
         7 rue Ren\'e Descartes\\
         67084 Strasbourg Cedex, France}
\email{vilmos.komornik@math.unistra.fr}

\author{Paola Loreti}
\address{Dipartimento di Metodi e Modelli\\
           Matematici per le Scienze Applicate\\
           Universit\`a degli Studi di Roma ``La Sapienza''\\
           Via A. Scarpa, 16\\
           00161 Roma, Italy}
\email{loreti@dmmm.uniroma1.it}
\thanks{}
%wave equation, observability, analyse non harmonique, transform\'ee de Fourier, base de Riesz, 
%\'equation des ondes, observabilit\'e}
%\thanks
%\dedicatory{Dedicated to ??? on the occasion of ???.}

\begin{abstract}
We formulate and discuss a conjecture which would extend  a classical inequality of Bernstein. 
\end{abstract}   
\maketitle

\section{A Bernstein type inequality}\label{s1}

A classical theorem due to Bernstein \cite{Ber1912} states that every even trigonometric polynomial $T$ of order $M$ satisfies the 
inequality
\begin{equation*}
\norm{T'}_{L^\infty}\le M\norm{T}_{L^\infty}.
\end{equation*}
His result was extended to all trigonometric polynomials by Fej\'er \cite{Fej1914}.
Another proof was given by M. Riesz \cite{Rie1914a}, \cite{Rie1914b}; this also shows that
\begin{equation}\label{11}
\norm{T'}_{L^p(I)}\le M\norm{T}_{L^p(I)}
\end{equation}
for every interval $I$ of length $\abs{I}=2\pi$ and for every exponent $1\le p\le\infty$. 

For $p=2$ this inequality follows easily by applying Parseval's formula. Indeed, writing
\begin{equation*}
T(x)=\sum_{k=-M}^{M}a_ke^{ikx}
\end{equation*}
and using the orthogonality of the functions $e^{ikx}$ we have
\begin{align*}
\int_{I}\abs{T'(x)}^2\ dx&-M^2\int_I\abs{T(x)}^2\ dx\\
&=\int_I\Bigl\vert\sum_{k=-M}^{M}ika_ke^{ikx}\Bigr\vert^2\ dx -M^2   \int_I\Bigl\vert\sum_{k=-M}^{M}a_ke^{ikx}\Bigr\vert^2\ dx\\
&= 2\pi\sum_{k=-M}^{M}\abs{ika_k}^2-2\pi M^2\sum_{k=-M}^{M}\abs{a_k}^2\\
&= 2\pi \sum_{k=-M}^{M}(k^2-M^2)\abs{a_k}^2\\
&\le 0.
\end{align*}

\section{A conjecture}\label{s2}

Let us introduce the function
\begin{equation*}
H(x):=
\begin{cases}
\cos x&\text{if $\abs{x}\le \pi/2$,}\\
0&\text{if $\abs{x}\ge \pi/2$.}
\end{cases}
\end{equation*}
For any positive integer $M$, the following inequality holds:
\begin{equation}\label{21}
\int _{-\infty}^\infty\abs{(H^M)'(x)}^2\ dx\le M^2 \int _{-\infty}^\infty\abs{H^M(x)}^2\ dx.
\end{equation}
Indeed, since
\begin{multline*}
\int _{-\infty}^\infty\abs{(H^M)'(x)}^2\ dx
=\int _{-\pi/2}^{\pi/2}\abs{(\cos^M)'(x)}^2\ dx\\
=\int_0^{\pi}\abs{(\sin^M)'(x)}^2\ dx
=\frac{1}{2}\int_{-\pi}^{\pi}\abs{(\sin^M)'(x)}^2\ dx
\end{multline*}
and
\begin{equation*}
 \int _{-\infty}^\infty\abs{H^M(x)}^2\ dx
=\int _{-\pi/2}^{\pi/2}\abs{\cos^Mx}^2\ dx
=\int_0^{\pi}\abs{\sin^M x}^2\ dx
=\frac{1}{2}\int_{-\pi}^{\pi}\abs{\sin^M x}^2\ dx,
\end{equation*}
the inequality follows by applying \eqref{11} with $T(x):=\sin^M x$ on the interval $I=(-\pi,\pi)$.

The following conjecture is a generalization of the inequality \eqref{21}. 

\begin{conjecture}
Let $(\lam _n)_{n=-\infty}^{\infty}$ be a strictly 
increasing sequence  of
real numbers, satisfying for some positive integer $M$ the {\em gap condition}
\begin{equation}\label{22}
\lam_{n+M}-\lam_n\ge \pi
\end{equation}
for all $n$. Then for every finite sequence $(a_n)$ of real numbers, the function
\begin{equation*}
G(x):=\sum a_nH^M(x+\lam_n)
\end{equation*}
satisfies the inequality
\begin{equation}\label{23}
\int _{-\infty}^\infty\abs{G'(x)}^2\ dx\le M^2 \int _{-\infty}^\infty\abs{G(x)}^2\ dx.
\end{equation}
\end{conjecture}

In the next sections we prove the conjecture for $M=1$ and $M=2$.

\section{Proof of the conjecture for $M=1$}\label{s3}

For $M=1$ we have even an equality. Indeed, since for $m\ne n$ the product functions 
\begin{equation*}
 H(x+\lam_m)H(x+\lam_n)\quad\text{and}\quad H'(x+\lam_m)H'(x+\lam_n)
\end{equation*}
vanish identically by \eqref{22}, we have
\begin{equation*}
\int _{-\infty}^\infty\abs{G(x)}^2\ dx
=\int _{-\infty}^\infty\Bigl\vert \sum a_nH(x+\lam_n)\Bigr\vert ^2\ dx
=\sum \abs{a_n}^2\int _{-\infty}^\infty\abs{H(x+\lam_n)}^2\ dx
\end{equation*}
and
\begin{equation*}
\int _{-\infty}^\infty\abs{G'(x)}^2\ dx
=\int _{-\infty}^\infty\Bigl\vert \sum a_nH'(x+\lam_n)\Bigr\vert ^2\ dx
=\sum \abs{a_n}^2\int _{-\infty}^\infty\abs{H'(x+\lam_n)}^2\ dx.
\end{equation*}
We conclude by observing that
\begin{equation*}
\int _{-\infty}^\infty\abs{H(x+\lam_n)}^2\ dx=\int _{-\pi/2}^{\pi/2}\cos^2 x\ dx=\frac{\pi}{2}
\end{equation*}
and
\begin{equation*}
\int _{-\infty}^\infty\abs{H'(x+\lam_n)}^2\ dx=\int _{-\pi/2}^{\pi/2}\sin^2 x\ dx=\frac{\pi}{2}.
\end{equation*}

\section{Discussion of the case $M\ge 2$}\label{s4}

We begin with some discussion concerning the general case. Our first lemma allows us to reformulate the conjecture.

\begin{lemma}\label{l41}
Introducing the function
\begin{equation*}
g(\lam)=g_M(\lam):=\int _{-\infty}^\infty H^M(x+\lam) H^{M-2}(x)\ dx,
\end{equation*}
we have
\begin{equation*}
M^2 \int _{-\infty}^\infty\abs{G(x)}^2\ dx-\int _{-\infty}^\infty\abs{G'(x)}^2\ dx
=\sum _{m,n=-\infty}^{\infty}g(\lam_m-\lam_n)a_m\overline{a_n}.
\end{equation*}
\end{lemma}

\begin{proof}
We recall from \cite{BaiKomLor103} that
\begin{equation*}
(H^M)''(x)=-M^2H^M(x)+M(M-1)H^{M-2}(x)
\end{equation*}
for all $x$. Integrating by parts and then using this relation, we have
\begin{align*}
M^2 &\int _{-\infty}^\infty\abs{G(x)}^2\ dx-\int _{-\infty}^\infty\abs{G'(x)}^2\ dx\\
&=\sum _{m,n=-\infty}^{\infty}a_m\overline{a_n}\int _{-\infty}^\infty
M^2H^M(x+\lam_m)H^M(x+\lam_n) \\
&\qquad\qquad\qquad\qquad\qquad\qquad -(H^M)'(x+\lam_m)(H^M)'(x+\lam_n)\ dx\\
&=\sum _{m,n=-\infty}^{\infty}a_m\overline{a_n}\int _{-\infty}^\infty
H^M(x+\lam_m)M(M-1)H^{M-2}(x+\lam_n)\ dx\\
&=M(M-1)\sum _{m,n=-\infty}^{\infty}g(\lam_m-\lam_n)a_m\overline{a_n}.\qedhere
\end{align*}
\end{proof}

In view of this lemma it suffices to show that 

\begin{equation}\label{41}
\sum _{m,n=-\infty}^{\infty}g(\lam_m-\lam_n)a_m\overline{a_n}\ge 0
\end{equation} 
for all finite sequences $(a_n)$ of complex numbers.

\begin{remark}
It follows easily from the definition that $g_M$ is a nonnegative, even function, vanishing outside the interval $(-\pi,\pi)$. It can  be computed explicitly for any given $M$. For example, if $0\le x\le \pi$, then we have
\begin{align*}
4g_2(x)&=2(\pi -x)+ \sin 2x,\\
32g_3(x)&=12(\pi-x)\cos x+9\sin x+\sin 3x,\\
192g_4(x)&=36(\pi -x) + 24 (\pi - x) \cos 2x + 28 \sin 2x + \sin 4x.
\end{align*}
Indeed, for $M=2$ we have
\begin{equation*}
4g_2(x)=\int_{-\pi/2}^{\pi/2-x}4\cos^2t\ dt=\int_{-\pi/2}^{\pi/2-x}2+2\cos 2t\ dt=2(\pi-x)+\sin 2x.
\end{equation*}
For $M=3,4$ the computation is similar but longer.
\end{remark}

\section{Proof of the conjecture for $M=2$}\label{s5}

The proof of \eqref{41} for $M=2$ is based on the following identity:

\begin{lemma}\label{l42}
The following identity holds:
\begin{align*}
\sum _{m,n=-\infty}^{\infty}g(\lam_m-\lam_n)a_m\overline{a_n}
&=\sum _{n=-\infty}^{\infty}g(\lam_{n+1}-\lam_n)\abs{a_n+a_{n+1}}^2\\
&\qquad +\sum _{n=-\infty}^{\infty}\bigl(g(0)-g(\lam_n-\lam_{n-1})-g(\lam_{n+1}-\lam_n)\bigr)
\abs{a_n}^2.
\end{align*}
\end{lemma}

\begin{proof}
Writing $h_{m,n}:=g(\lam_m-\lam_n)$ for brevity, and using the evenness of $g$,  the following computation leads to the 
required identity:
\begin{align*}
\sum_{m,n=-\infty}^{\infty}&h_{m,n}a_m\overline{a_n}\\
&=\sum_{n=-\infty}^{\infty}h_{n,n}\abs{a_n}^2+h_{n,n+1}(a_n\overline{a_{n+1}}+\overline{a_n}a_{n+1})\\
&=\sum_{n=-\infty}^{\infty}h_{n,n}\abs{a_n}^2
+h_{n,n+1}\bigl(\abs{a_n+a_{n+1}}^2-\abs{a_n}^2-\abs{a_{n+1}}^2\bigr)\\
&=\sum_{n=-\infty}^{\infty}h_{n,n+1}\abs{a_n+a_{n+1}}^2
+(h_{n,n}-h_{n,n+1}-h_{n-1,n})\abs{a_n}^2.\qedhere
\end{align*}
\end{proof}

Since $g$ is nonnegative, the first sum on the right side of the above identity is $\ge 0$. Since 
\begin{equation*}
\lam_{n+1}-\lam_n\ge 0,\quad \lam_n-\lam_{n-1}\ge 0\quad\text{and}\quad (\lam_{n+1}-\lam_n)+(\lam_n-\lam_{n-1})=\lam_{n+1}-\lam_{n-1}\ge \pi
\end{equation*}
by the gap condition \eqref{22}, the nonnegativity of the second sum follows from the next lemma which completes the proof of \eqref{41}.

\begin{lemma}\label{l43}
If 
\begin{equation*}
a\ge 0,\quad b\ge 0\quad \text{and}\quad a+b\ge\pi,
\end{equation*}
then
\begin{equation*}
g(a)+g(b)\le g(0).
\end{equation*}
\end{lemma}

\begin{proof}
Since the functions $H^2$, $H^0$ are nonnegative and since the intervals
\begin{equation*}
\Bigl(\frac{-\pi}{2},\frac{\pi}{2}-a\Bigr)\quad \text{and}\quad 
\Bigl(\frac{-\pi}{2}+b,\frac{\pi}{2}\Bigr)
\end{equation*}
are disjoint, we have
\begin{align*}
g(a)+g(b)&=\int _{-\infty}^\infty H^2(x) H^0(x+a)\ dx
+\int _{-\infty}^\infty H^2(x) H^0(x-b)\ dx\\
&=\int _{\frac{-\pi}{2}}^{\frac{\pi}{2}-a}H^2(x)\ dx
+\int _{\frac{-\pi}{2}+b}^{\frac{\pi}{2}}H^2(x)\ dx\\
&\le \int _{-\infty}^\infty H^2(x) \ dx\\
&=g(0).\qedhere
\end{align*}
\end{proof}

\begin{remark}
The above proof also shows that for $a+b=\pi$ we have $g(a)+g(b)=g(0)$, i.e.,
\begin{equation*}
g(x)+g(\pi -x)= g(0)
\end{equation*}
for all $x\in [0,\pi]$. This can also be seen from the explicit formula of $g$.
\end{remark}

\end{document}